\documentclass[12pt]{article}
\usepackage{color}
\usepackage{graphicx}
\usepackage{colortbl}
\usepackage{array}
\usepackage{amssymb}
\usepackage{amsmath}
\usepackage{mathrsfs}
\usepackage{dcolumn}
\usepackage{longtable}
\usepackage{hhline}
\usepackage{graphics}
\usepackage{amssymb}
\usepackage{amsmath}
\usepackage{amsthm}
\usepackage{mathrsfs}
\usepackage{amsfonts}
\usepackage{latexsym}
\usepackage{color}
\usepackage{epsf}

\newtheorem{thm}{Theorem}[section]
\newtheorem{thrm}{Theorem}[section]

\newtheorem{deftn}[thrm]{Definition}

\newtheorem{rema}[thm]{Remark}

\newtheorem{corr}[thrm]{Corollary}

\newcommand{\rsts}[2]{\genfrac{\{}{\}}{0pt}{0}{#1}{#2}}
\newcommand{\rstf}[2]{\genfrac{[}{]}{0pt}{0}{#1}{#2}}

\title{\textbf{Higher Order Apostol-Type Poly-Genocchi Polynomials with Parameters $a, b$ and $c$}}
\author{\textbf{Cristina B. Corcino}\\\textbf{Roberto B. Corcino}\\{Research Institute for Computational} {Mathematics and Physics}\\Mathematics Department\\Cebu Normal University\\Cebu City, Philippines
}

\begin{document}
\maketitle
\begin{abstract}
In this paper, a new form of poly-Genocchi polynomials is defined by means of polylogarithm, namely, the Apostol-type poly-Genocchi polynomials of higher order with parameters $a$, $b$ and $c$.  Several properties of these polynomials are established including some recurrence relations and explicit formulas, which express these higher order Apostol-type poly-Genocchi polynomials in terms of Stirling numbers of the second kind, Apostol-type Bernoulli and Frobenius polynomials of higher order. Moreover, certain differential identity is obtained that leads this new form of poly-Genocchi polynomials to be classified as Appell polynomials and, consequently, draw more properties using some theorems on Appell polynomials. Furthermore, a symmetrized generalization of this new form of poly-Genocchi polynomials is introduced that possesses a double generating function. Finally, the type 2 Apostol-poly-Genocchi polynomials with parameters $a$, $b$ and $c$ are defined using the concept of polyexponential function and several identities are derived, two of which show the connections of these polynomials with Stirling numbers of the first kind and the type 2 Apostol-type poly-Bernoulli polynomials.


\end{abstract}

\section{Introduction}
There are several variations of Genocchi numbers that appeared in the literature. These include the Genocchi polynomials and Genocchi polynomials of higher order, which are respectively defined by
\begin{align}
\sum_{n=0}^{\infty}G_n(x)\frac{t^n}{n!}&=\frac{2t}{e^t+1}e^{xt},\;\;\; |t|<\pi,\label{eqn1}\\
\sum_{n=0}^{\infty}G_n^{(k)}(x)\frac{t^n}{n!}&=\left(\frac{2t}{e^t+1}\right)^k e^{xt},\label{eqn2}
\end{align}
and Apostol-Genocchi polynomials, and Apostol-Genocchi polynomials of higher order, which are respectively defined by
\begin{align}
\sum_{n=0}^{\infty}G_n(x,\lambda)\frac{t^n}{n!}&=\frac{2t}{\lambda e^t+1} e^{xt},\label{eqn3}\\
\sum_{n=0}^{\infty}G_n^{(k)}(x,\lambda)\frac{t^n}{n!}&=\left(\frac{2t}{\lambda e^t+1}\right)^k e^{xt},\label{eqn4}
\end{align}
 (see \cite{ref8,ref9,ref36,ref37,ref46-1,ref47,ref49}). Another variation of Genocchi numbers, also known as poly-Genocchi polynomials, was introduced by Kim et al. \cite{ref46} using the concept of $kth$ polylogarithm, denoted by ${\rm Li}_k(z)$, which is given by
\begin{equation}\label{polylog}
{\rm Li}_k(z)=\sum_{n=0}^{\infty}\frac{z^n}{n^k}, k\in\mathbb{Z}.
\end{equation}
More precisely, poly-Genocchi polynomials are defined as follows
\begin{equation}\label{polyGeno1}
		\sum_{n=0}^{\infty}G_n^{(k)}(x)\frac{x^n}{n!}=\frac{2Li_{k}(1-e^{t})}{e^{t}+1}e^{xt}.
\end{equation}
Moreover, a modified poly-Genocchi polynomials, denoted by $G_{n,2}^{(k)}(x)$, were defined by Kim et al. \cite{ref46} as follows
\begin{equation}\label{polyGeno2}
		\sum_{n=0}^{\infty}G_{n,2}^{(k)}(x)\frac{x^n}{n!}=\frac{Li_{k}(1-e^{-2t})}{e^{t}+1}e^{xt}.
\end{equation} 
Note that, when $k=1$, equations \eqref{polyGeno1} and \eqref{polyGeno2} give the Genocchi polynomials in \eqref{eqn1}. That is, 
$$G_n^{(1)}(x)=G_{n,2}^{(1)}(x)=G_{n}(x).$$
Kim et. al \cite{ref46} obtained several properties of these polynomials. 

\smallskip
On the other hand, Kurt \cite{ref51} defined two forms of generalized poly-Genocchi polynomials with parameters $a$, $b$, and $c$, as follows
\begin{align}
		\frac{2Li_{k}(1-(ab)^{-t})}{a^{-t}+b^{t}}e^{xt}&=\sum_{n=0}^{\infty}G_n^{(k)}(x;a,b,c)\frac{x^n}{n!}\label{polyGeno11}\\
		\frac{2Li_{k}(1-(ab)^{-2t})}{a^{-t}+b^{t}}e^{xt}&=\sum_{n=0}^{\infty}G_{n,2}^{(k)}(x;a,b,c)\frac{x^n}{n!},\label{polyGeno22}
\end{align}
which are motivated by the definitions in \eqref{polyGeno1} and \eqref{polyGeno2}, respectively. Kurt \cite{ref51} also derived several properties parallel to those of poly-Genocchi polynomials by Kim et al. \cite{ref46}. Note that, when $x=0$, \eqref{polyGeno1} reduces to
\begin{equation}\label{polyGenonum}
		\frac{2Li_{k}(1-e^{t})}{e^{t}+1}=\sum_{n=0}^{\infty}G_n^{(k)}\frac{x^n}{n!}.
\end{equation}
where $G_n^{(k)}$ are called the poly-Genocchi numbers. 

\smallskip
In this paper, a new variation of poly-Genocchi polynomials with parameters $a$, $b$ and $c$, namely, the Apostol-type poly-Genocchi polynomials of higher order with parameters $a, b$ and $c$, will be investigated. Sections 2 and 3 provide the definition of this new variation of poly-Genocchi polynomials, some special cases and their relations with some Genocchi-type polynomials. Section 4 devotes its discussion on some identities that link the Apostol-type poly-Genocchi polynomials of higher order with parameters $a, b$ and $c$ to Appell polynomials. Section 5 focuses on the connections of these higher order Apostol-type polynomials to Stirling numbers of the second kind and different variations of higher order Bernoulli-type polynomials. Section 6 demonstrates other form of generalization of these higher order Apostol-type polynomials, the symmetrized generalization. Section 7 introduces type 2 Apostol-poly-Genocchi polynomials with parameters $a, b$ and $c$ using the concept of polyexponential function \cite{ref45-1}. Section 8 contains the conclusion of the paper.

\section{Definition}

Now, a new variation of poly-Genocchi polynomials, the Apostol-Type poly-Genocchi polynomials of higher order with parameters $a$ $b$ and $c$ can be introduced.
\begin{deftn}\rm The Apostol-type poly-Genocchi polynomials of higher order with parameters $a, b$ and $c$, denoted by $\mathcal{G}_n^{(k,\alpha)}(x; \lambda, a, b, c)$, are defined as follows
\begin{equation}\label{eq711}
\sum_{n=0}^{\infty}\mathcal{G}_n^{(k,\alpha)}(x; \lambda, a, b, c)\frac{t^n}{n!}=\left(\frac{Li_{k}(1-(ab)^{-2t})}{a^{-t}+\lambda b^t}\right)^{\alpha}c^{xt},\;\;\; |t| < \frac{\sqrt{(\ln \lambda)^2+\pi^2}}{|\ln a + \ln b|}.
\end{equation}
When $\alpha=1$, \eqref{eq711} yields
\begin{equation}\label{eq711-1}
\sum_{n=0}^{\infty}\mathcal{G}_n^{(k)}(x; \lambda, a, b, c)\frac{t^n}{n!}=\frac{Li_{k}(1-(ab)^{-2t})}{a^{-t}+\lambda b^t}c^{xt},\;\;\; |t| < \frac{\sqrt{(\ln \lambda)^2+\pi^2}}{|\ln a + \ln b|},
\end{equation}
where $\mathcal{G}_n^{(k)}(x; \lambda, a, b, c)=\mathcal{G}_n^{(k,1)}(x; \lambda, a, b, c)$ denotes the Apostol-type poly-Genocchi polynomials with parameters $a, b$ and $c$.
\end{deftn}

\smallskip
The following are some special cases of the Apostol-type poly-Genocchi polynomials of higher order with parameters $a, b$ and $c$:
\begin{enumerate}
\item When $c=e$, equation \eqref{eq711} reduces to
\begin{equation}\label{eq13}
\sum_{n=0}^{\infty}\mathcal{G}_n^{(k,\alpha)}(x;\lambda, a, b, e) \frac{t^n}{n!}=\left(\frac{Li_{k}(1-(ab)^{-2t})}{a^{-t}+\lambda b^{t}}\right)^{\alpha} e^{xt}.
\end{equation}
For convenience, we use $\mathcal{G}_n^{(k,\alpha)}(x; \lambda, a, b)$ to denote $\mathcal{G}_n^{(k,\alpha)}(x; \lambda, a, b, e)$. That is,
\begin{equation}\label{eq-2}
\sum_{n=0}^{\infty}\mathcal{G}_n^{(k,\alpha)}(x; \lambda, a, b) \frac{t^n}{n!}=\left(\frac{Li_{k}(1-(ab)^{-2t})}{a^{-t}+\lambda b^{t}}\right)^{\alpha} e^{xt}.
\end{equation}
\item When $k=1$, \eqref{eq711} yields
\begin{equation}\label{eqq711}
\sum_{n=0}^{\infty}\mathcal{G}_n^{(\alpha)}(x; \lambda, a, b, c)\frac{t^n}{n!}=\left(\frac{2t\ln ab}{a^{-t}+\lambda b^t}\right)^{\alpha}c^{xt},\;\;\; |t| < \frac{\sqrt{(\ln \lambda)^2+\pi^2}}{|\ln a + \ln b|},
\end{equation}
where the polynomials $\mathcal{G}_n^{(\alpha)}(x; \lambda, a, b, c)=\mathcal{G}_n^{(1,\alpha)}(x; \lambda, a, b, c)$ are called the Apostol-type Genocchi polynomials of higher order with parameters $a$ $b$ and $c$. When $\alpha=1$, \eqref{eqq711} yields
\begin{equation}\label{eqq711-1}
\sum_{n=0}^{\infty}\mathcal{G}_n^{(1)}(x; \lambda, a, b, c)\frac{t^n}{n!}=\frac{2t\ln ab}{a^{-t}+\lambda b^t}c^{xt},\;\;\; |t| < \frac{\sqrt{(\ln \lambda)^2+\pi^2}}{|\ln a + \ln b|},
\end{equation}
where the polynomials $\mathcal{G}_n^{(1)}(x; \lambda, a, b, c)$ are called the Apostol-type Genocchi polynomials with parameters $a$ $b$ and $c$. 
\item When $a=1,b=e$, $\eqref{eq-2}$ will reduce to
\begin{equation}
\sum_{n=0}^{\infty}\mathcal{G}_n^{(k,\alpha)}(x;\lambda, 1, e) \frac{t^n}{n!}=\left(\frac{Li_{k}(1-e^{-2t})}{1+\lambda e^{t}}\right)^{\alpha} e^{xt}.\label{ApostolGenoPoly}
\end{equation}
We may use the notations 
$$\mathcal{G}_n^{(k,\alpha)}(x;\lambda)=\mathcal{G}_n^{(k,\alpha)}(x;\lambda, 1, e)\;\;\mbox{and}\;\; \mathcal{G}_n^{(k)}(x;\lambda)=\mathcal{G}_n^{(k)}(x;\lambda, 1, e)$$ 
and call them Apostol-type poly-Genocchi polynomials of higher order and Apostol-type poly-Genocchi polynomials, respectively. 
\item When $\lambda=1$, \eqref{ApostolGenoPoly} gives 
\begin{equation}
\sum_{n=0}^{\infty}\mathcal{G}_n^{(k,\alpha)}(x;1) \frac{t^n}{n!}=\left(\frac{Li_{k}(1-e^{-2t})}{1+e^{t}}\right)^{\alpha} e^{xt}.\label{ApostolGenPoly}
\end{equation}
which is the higher order version of equation \eqref{polyGeno2}, i.e. the higher order version of the modified poly-Genocchi polynomials of Kim et al.\cite{ref46}. We may use $G_{n,2}^{(k,\alpha)}(x)$ to denote $\mathcal{G}_n^{(k,\alpha)}(x;1)$. 
\item Using the fact that 
$$Li_1(z)=-\ln(1-z),$$
when $k=1$, \eqref{ApostolGenoPoly} gives 
\begin{equation}
\sum_{n=0}^{\infty}\mathcal{G}_n^{(1,\alpha)}(x;\lambda) \frac{t^n}{n!}=\left(\frac{2t}{1+\lambda e^{t}}\right)^{\alpha} e^{xt},\label{ApostolGenPoly1}
\end{equation}
and when $\lambda=1$, \eqref{ApostolGenPoly1} gives 
\begin{equation*}
\sum_{n=0}^{\infty}\mathcal{G}_n^{(1,\alpha)}(x;1) \frac{t^n}{n!}=\left(\frac{2t}{1+e^{t}}\right)^{\alpha} e^{xt}.
\end{equation*}
where $\mathcal{G}_n^{(1,\alpha)}(x;\lambda)$ and $\mathcal{G}_n^{(1,\alpha)}(x;1)$ are exactly the Genocchi polynomials $G_n^{(k)}(x)$ and Apostol-Genocchi polynomials $G_n^{(k)}(x,\lambda)$ of higher order in \eqref{eqn4} and \eqref{eqn2}, respectively. 
\end{enumerate}

\section{Relations with Some Genocchi-type Polynomials}

In this section, some relations for $\mathcal{G}_n^{(k,\alpha)}(x;\lambda,a,b,c)$ will be established which are expressed in terms of some Genocchi-type polynomials.

\smallskip
First, certain recurrence relation for $\mathcal{G}_n^{(k,\alpha)}(x;\lambda,a,b,c)$ will be established.

\begin{thrm} The Apostol-type poly-Genocchi polynomials of higher order with parameters $a, b, c$ satisfy the recurrence relation
\begin{equation}\label{eq33}
\mathcal{G}_n^{(k,\alpha)}(x+1;\lambda,a,b,c)=\sum_{r=0}^n\binom{n}{r}(\ln c)^r\mathcal{G}_{n-r}^{(k,\alpha)}(x;\lambda,a,b,c).
\end{equation}
\end{thrm}

\begin{proof} Equation \eqref{eq711} can be written as
\begin{align*}
\sum_{n=0}^{\infty}\mathcal{G}_n^{(k,\alpha)}(x+1;\lambda,a,b,c) \frac{t^n}{n!}&=\left(\frac{Li_{k}(1-(ab)^{-2t})}{a^{-t}+\lambda b^{t}}\right)^{\alpha} e^{xt\ln c}e^{t\ln c}\\
&=\left\{\sum_{n=0}^{\infty}\mathcal{G}_n^{(k,\alpha)}(x;\lambda,a,b,c) \frac{t^n}{n!}\right\}\left\{\sum_{n=0}^{\infty}\frac{(t\ln c)^n}{n!}\right\}\\
&=\sum_{n=0}^{\infty}\sum_{r=0}^n\mathcal{G}_{n-r}^{(k,\alpha)}(x;\lambda,a,b,c)\frac{t^{n-r}}{(n-r)!}\frac{(\ln c)^rt^r}{r!}\\
&=\sum_{n=0}^{\infty}\left\{\sum_{r=0}^n\binom{n}{r}(\ln c)^r\mathcal{G}_{n-r}^{(k,\alpha)}(x;\lambda,a,b,c)\right\}\frac{t^n}{n!}.
\end{align*}
Comparing the coefficients of $\frac{t^n}{n!}$ completes the proof of the theorem.
\end{proof}

\smallskip
Consider a special case of \eqref{eq-2} by taking $x=0$. This gives
\begin{equation}\label{polyGenocchiNumbers}
\sum_{n=0}^{\infty}\mathcal{G}_n^{(k,\alpha)}(0;\lambda, a, b) \frac{t^n}{n!}=\left(\frac{Li_{k}(1-(ab)^{-2t})}{a^{-t}+\lambda b^{t}}\right)^{\alpha}.
\end{equation}
We use the notation $\mathcal{G}_{i}^{(k,\alpha)}(\lambda,a,b)=\mathcal{G}_{i}^{(k,\alpha)}(0;\lambda,a,b)$ and call them the Apostol-type poly-Genocchi numbers of higher order with parameters $a$ and $b$. The following theorem contains an identity that expresses $\mathcal{G}_n^{(k,\alpha)}(x;\lambda, a, b, c)$ as polynomial in $x$, which involves $\mathcal{G}_{i}^{(k,\alpha)}(\lambda,a,b)$ as coefficients.
	
\smallskip
\begin{thrm}\label{thm32} The Apostol-type poly-Genocchi polynomials of higher order with parameters $a, b, c$ satisfy the relation, 
\begin{equation}\label{eq9}
 \mathcal{G}_n^{(k,\alpha)}(x;\lambda, a, b, c)=\sum_{i=0}^{n} \binom {n}{i} (\ln c)^{n-i} \mathcal{G}_{i}^{(k,\alpha)}(\lambda,a,b) x^{n-i}.
\end{equation}
\end{thrm}
  
\begin{proof} 
Equation \eqref{eq711} can be written as 
\begin{align*}
\sum_{n=0}^{\infty} \mathcal{G}_n^{(k,\alpha)}(x; a, b, c)\dfrac{t^n}{n!} &=\left(\dfrac{Li_{k}(1-(ab)^{-2t})}{a^{-t}+\lambda b^t}\right)^{\alpha}c^{xt} =e^{xt \ln c} \sum_{n=0}^{\infty}\mathcal{G}_n^{(k,\alpha)}(\lambda,a,b)\dfrac{t^n}{n!} \\   
&=\sum_{n=0}^{\infty} \sum_{i=0}^{\infty} \dfrac{(xt \ln c)^{n-i}}{(n-i)!} \mathcal{G}_i^{(k,\alpha)} (\lambda,a,b) \dfrac {t^i}{i!} \\ 
&=\sum_{n=0}^{\infty} \left(\ \sum_{i=0}^{\infty} \binom{n}{i} (\ln c)^{n-i} \mathcal{G}_i^{(k,\alpha)} (\lambda,a,b) x^{n-i} \right)\ \frac {t^n}{n!} 
\end{align*}
Comparing the coefficients of $\frac{t^n}{n!}$, we obtain the desired result. 
\end{proof} 

The next identity gives the relation between $\mathcal{G}_n^{(k,\alpha)}(x;\lambda, a, b, c)$ and $\mathcal{G}_n^{(k,\alpha)} (x;\lambda)$.
 
\begin{thrm}\label{thm33} The Apostol-type poly-Genocchi polynomials of higher order with parameters $a, b, c$ satisfy the relation, 
\begin{equation}\label{eq10}
 \mathcal{G}_n^{(k,\alpha)}(x;\lambda, a, b, c)=(\ln a + \ln b)^{n} \mathcal{G}_n^{(k,\alpha)} \left(\ \frac{x \ln c + \alpha\ln a}{\ln a + \ln b};\lambda  \right).
\end{equation}
\end{thrm}

\begin{proof}

Using \eqref{eq711}, we have

\begin{align*}
\sum_{n=0}^{\infty} \mathcal{G}_n^{(k,\alpha)}(x;\lambda, a, b, c)\frac{t^n}{n!} &=\left(\frac{Li_{k}(1-(ab)^{-2t})}{a^{-t}(1+\lambda (ab)^t)}\right)^{\alpha}e^{xt \ln c} \\   
&= e^{\frac{x \ln c +\alpha\ln a}{\ln ab}t \ln ab} \left(\frac{Li_{k}(1-e^{-2t \ln ab})}{1+\lambda e^{t \ln ab}}\right)^{\alpha}  \\ 
&=\sum_{n=0}^{\infty} (\ln a + \ln b)^{n} \mathcal{G}_n^{(k,\alpha)} \left(\ \frac {x \ln c + \alpha\ln a}{\ln a + \ln b};\lambda  \right)\ \frac {t^n}{n!}. 
\end{align*}

\noindent Comparing the coefficients of $\frac{t^n}{n!}$, we obtain the desired result. 
\end{proof} 

\section{Classification as Appell Polynomials}

\smallskip
The following theorem contains a differential identity that can be used to classify Apostol-type poly-Genocchi polynomials as Appell polynomials \cite{ref51-1, ref57-1, ref57-2}

\begin{thrm}\label{thm34} The Apostol-type poly-Genocchi polynomials with parameters $a, b, c$ satisfy the relation, 
\begin{equation}\label{eq11}
 \frac {d}{dx} \mathcal{G}_{n+1}^{(k,\alpha)}(x;\lambda, a, b, c)= (n+1)(\ln c) \mathcal{G}_n^{(k,\alpha)} (x;\lambda,a,b,c).
\end{equation}
\end{thrm}

\begin{proof}
Applying the first derivative to equation \eqref{eq711}, we have
\begin{align*}
\sum_{n=0}^{\infty} \frac{d}{dx} \mathcal{G}_n^{(k,\alpha)}(x;\lambda, a, b, c)\frac{t^n}{n!} &= t(\ln c) \left(\frac{Li_{k} (1-(ab)^{-2t})}{(a^{-t}+\lambda b^{t})}\right)^{\alpha} e ^{xt \ln c} \\      
\sum_{n=0}^{\infty} \frac{d}{dx} \mathcal{G}_n^{(k,\alpha)}(x;\lambda, a, b, c)\frac{t^{n-1}}{n!} &= \sum_{n=0}^{\infty} (\ln c) \mathcal{G}_n^{(k,\alpha)}(x;\lambda, a, b, c)\frac{t^n}{n!} 
\end{align*}

Hence, 
$$\sum_{n=0}^{\infty} \frac{1}{n+1} \frac{d}{dx} \mathcal{G}_{n+1}^{(k,\alpha)}(x;\lambda, a, b, c) \frac{t^n}{n!} = \sum_{n=0}^{\infty} (\ln c) \mathcal{G}_n^{(k,\alpha)}(x;\lambda, a, b, c)\frac{t^n}{n!}$$.

\noindent Comparing the coefficients of $\frac{t^n}{n!}$, we obtain the desired result. 
\end{proof} 

\begin{rema}\rm
When $c=e$, equation \eqref{eq11} reduces to
\begin{equation}\label{eq111}
 \frac {d}{dx} \mathcal{G}_{n+1}^{(k,\alpha)}(x;\lambda, a, b)= (n+1) \mathcal{G}_n^{(k,\alpha)} (x;\lambda,a,b),
\end{equation}
which is one of the property for the polynomial to be classified as Appell polynomial. 
\end{rema}

\smallskip
Being classified as Appell polynomials, the generalized poly-Genocchi polynomials $\mathcal{G}^{(k)}_{n}(x;a,b)$ must possess the following properties
\begin{align*}
\mathcal{G}^{(k,\alpha)}_{n}(x;\lambda,a,b)&=\sum_{i=0}^n\binom{n}{i}c_ix^{n-i}\\
\mathcal{G}^{(k,\alpha)}_{n}(x;\lambda,a,b)&=\left(\sum_{i=0}^{\infty}\frac{c_i}{i!}D^i\right)x^{n},
\end{align*}
for some scalar $c_i\neq0$. It is then necessary to find the sequence $\{c_n\}$. However, by using \eqref{eq9} with $c=e$, $c_i=\mathcal{G}^{(k,\alpha)}_i(\lambda,a,b)$. This implies the following corollary.

\bigskip
\begin{corr}\label{genthm33} 
The Apostol-type poly-Genocchi polynomials with parameters $a, b, c$ satisfy the formula,
\begin{equation*}
\mathcal{G}^{(k,\alpha)}_{n}(x;\lambda,a,b)=\left(\sum_{i=0}^{\infty}\frac{\mathcal{G}^{(k,\alpha)}_i(\lambda,a,b)}{i!}D^i\right)x^{n}.
\end{equation*}
\end{corr}
For example, when $n=3$, we have
\begin{align*}
\mathcal{G}^{(k,\alpha)}_{3}(x;\lambda,a,b)&=\left(\sum_{i=0}^{\infty}\frac{\mathcal{G}^{(k,\alpha)}_i(\lambda,a,b)}{i!}D^i\right)x^{3}\\
&=\frac{\mathcal{G}^{(k,\alpha)}_0(\lambda,a,b)}{0!}x^3+\frac{\mathcal{G}^{(k,\alpha)}_1(\lambda,a,b)}{1!}D^1 x^3+\frac{\mathcal{G}^{(k,\alpha)}_2(\lambda,a,b)}{2!}D^2 x^3\\
&\;\;\;\;\;\;+\frac{\mathcal{G}^{(k,\alpha)}_3(\lambda,a,b)}{3!}D^3 x^3\\
&=\mathcal{G}^{(k,\alpha)}_0(\lambda,a,b)x^3+3\mathcal{G}^{(k,\alpha)}_1(\lambda,a,b) x^2+3\mathcal{G}^{(k,\alpha)}_2(\lambda,a,b) x+\mathcal{G}^{(k,\alpha)}_3(\lambda,a,b).
\end{align*}
The next corollary immediately follows from equation \eqref{eq111} and the characterization of Appell polynomials \cite{ref51-1, ref57-1, ref57-2}. 

\begin{corr}\label{cor3-5}
The Apostol-type poly-Genocchi polynomials with parameters $a, b, c$ satisfy the addition formula
	\begin{equation}\label{eq012}
		\mathcal{G}_n^{(k,\alpha)} (x+y;\lambda,a,b)=\sum_{i=0}^{\infty} \binom{n}{i} \mathcal{G}_i^{(k,\alpha)} (x;\lambda,a,b) y^{n-i}.
	\end{equation}
\end{corr}

\indent Taking $x=0$ in formula \eqref{eq012} and using the fact $\mathcal{G}_n^{(k)}(0;\lambda,a,b)=\mathcal{G}_n^{(k)}(\lambda,a,b)$, Corollary \ref{cor3-5} gives formula \eqref{eq9} in Theorem \ref{thm32} with $c=e$.

An extension of this addition formula can be derived as follows:
\begin{align*}
\sum_{n=0}^{\infty} \mathcal{G}_n^{(k,\alpha)}(x+y;\lambda, a, b, c)\frac{t^n}{n!} &= \left(\frac{Li_{k} (1-(ab)^{-2t})}{(a^{-t}+\lambda b^{t})}\right)^{\alpha} c^{xt}e^{yt\ln c} \\      
&= \left(\sum_{n=0}^{\infty} \mathcal{G}_n^{(k,\alpha)}(x;\lambda, a, b, c)\frac{t^n}{n!}\right)\left(\sum_{n=0}^{\infty} (y\ln c)^n\frac{t^n}{n!} \right) \\
&=\sum_{n=0}^{\infty}\left(\sum_{i=0}^n \binom{n}{i}\mathcal{G}_i^{(k,\alpha)}(x;\lambda, a, b, c)(y\ln c)^{n-i}\right)\frac{t^n}{n!}. 
\end{align*}
Comparing the coefficients of $\frac{t^n}{n!}$ gives the following theorem.

\begin{thrm}\label{thm36} The Apostol-type poly-Genocchi polynomials of higher order with parameters $a$, $b$ and $c$ satisfy the addition formula,
	\begin{equation}\label{eq12}
		\mathcal{G}_n^{(k,\alpha)} (x+y;\lambda,a,b,c)=\sum_{i=0}^{\infty} \binom{n}{i} (\ln c)^{n-i}\mathcal{G}_i^{(k,\alpha)} (x;\lambda,a,b,c) y^{n-i}.
	\end{equation}
\end{thrm}

\smallskip
By taking $x=0$, equation \eqref{eq12} exactly gives \eqref{eq9}.

\section{Connections with Some Special Numbers and Polynomials}

In this section, some connections of the higher order Apostol-type poly-Genocchi polynomials $\mathcal{G}_{n}^{(k,\alpha)}(x;\lambda,a,b,c)$ with other well-known special numbers and polynomials will be established.

\smallskip
To introduce the first connection, we need to define an Apostol-type poly-Bernoulli polynomials of higher order with parameters $a$, $b$ and $c$  as follows,
\begin{equation}\label{ApostolPolyBernoulli}
\left(\frac{Li_{k}(1-e^{-t})}{\lambda e^{t}-1}\right)^{\alpha}e^{xt}=\sum_{n=0}^{\infty}\mathcal{B}_n^{(k,\alpha)}(x;\lambda)\frac{x^n}{n!}.
\end{equation}
When $\alpha=1$, \eqref{ApostolPolyBernoulli} reduces to
\begin{equation}\label{ApostolPolyBernoulli-1}
\left(\frac{Li_{k}(1-e^{-t})}{\lambda e^{t}-1}\right)^{\alpha}e^{xt}=\sum_{n=0}^{\infty}\mathcal{B}_n^{(k)}(x;\lambda)\frac{x^n}{n!},
\end{equation}
where $\mathcal{B}_n^{(k)}(x;\lambda)=\mathcal{B}_n^{(k,1)}(x;\lambda)$ denotes the Apostol-type poly-Bernoulli polynomials with parameters $a$, $b$ and $c$.  When $k=1$, \eqref{ApostolPolyBernoulli} gives
\begin{equation}\label{ApostolBernoulli}
\left(\frac{t}{\lambda e^{t}-1}\right)^{\alpha}e^{xt}=\sum_{n=0}^{\infty}\mathcal{B}_n^{(1,\alpha)}(x;\lambda)\frac{x^n}{n!},
\end{equation}
where $\mathcal{B}_n^{(1,\alpha)}(x;\lambda)={B}^{(\alpha)}_n(x;\lambda)$, the Apostol-Bernoulli polynomials of higher order in \cite{ref54-1}. Also, when $\lambda=1$, \eqref{ApostolPolyBernoulli} will give 
\begin{equation}\label{PolyBernoulli}
\left(\frac{Li_{k}(1-e^{-t})}{e^{t}-1}\right)^{\alpha}e^{xt}=\sum_{n=0}^{\infty}\mathcal{B}_n^{(k,\alpha)}(x;1)\frac{x^n}{n!},
\end{equation}
where $\mathcal{B}_n^{(k,\alpha)}(x;1)=B_n^{(k,\alpha)}(x)$, the higher order version of poly-Bernoulli polynomials of Bayad and Hamahata \cite{ref16,ref47-1}. When $\alpha=1$, \eqref{PolyBernoulli} reduces to the definition of poly-Bernoulli numbers and polynomials \cite{ref41,ref42,ref43,ref47-1}. 

\smallskip
Now, we are ready to introduce the following theorem.

\begin{thrm} The Apostol-type poly-Genocchi polynomials of higher order with parameters $a, b, c$ satisfy the relation
\begin{align}
\mathcal{G}_{n}^{(k,\alpha)}(x;\lambda,a,b,c)&=\sum_{j=0}^{\alpha}\binom{\alpha}{j}(-1)^j\lambda^{\alpha-j} \mathcal{B}_{n}^{(k,\alpha)} \left( \frac {(\alpha-j)\ln b+x\ln c+(2\alpha-j)\ln a}{2(\ln a+\ln b)};\lambda^2 \right)\nonumber\\
&\;\;\;\;\;\;\;\;\;\;\times 2^{n}(\ln a+\ln b)^{n}.\label{eq55} 
\end{align}
In particular, the Apostol-type poly-Genocchi polynomials with parameters $a, b, c$ satisfy the relation,
\begin{align}
\mathcal{G}_{n}^{(k)}(x;\lambda,a,b,c)&=\left\{\lambda \mathcal{B}_{n}^{(k)} \left( \frac {\ln b+2\ln a+x\ln c}{2(\ln a+\ln b)};\lambda^2 \right)\right.\nonumber\\
&\;\;\;\;\;\;\;\left.-\mathcal{B}_{n}^{(k)} \left( \frac {\ln a+x\ln c}{2(\ln a+\ln b)};\lambda^2 \right) \right\} 2^{n}(\ln a+\ln b)^{n}.\label{eq13-o1} 
\end{align}
\end{thrm}

\begin{proof} Rewrite equation \eqref{eq711} as
\begin{align*}
\sum_{n=0}^{\infty} \mathcal{G}_{n}^{(k,\alpha)}(x;\lambda,a,b,c) \frac{t^n}{n!}&=\left(\frac{Li_{k}(1-(ab)^{-2t})}{a^{-t}+\lambda b^{t}}\right)^{\alpha} e^{xt \ln c}\\
&=\left(\frac{Li_{k}(1-(ab)^{-2t})}{a^{-2t}-(\lambda b^{t})^2}(a^{-t}-\lambda b^{t})\right)^{\alpha} e^{xt \ln c}\\
&=\left(\frac{Li_{k}(1-(ab)^{-2t})}{(1-(\lambda (ab)^{t})^2)}\right)^{\alpha}(e^{-t\ln a }-\lambda e^{t \ln b})^{\alpha}e^{xt \ln c} e^{2\alpha t\ln a}\\
&=\left(\frac{Li_{k}(1-e^{-2t(\ln ab)})}{-(\lambda e^{2t\ln (ab)}-1)}\right)^{\alpha}\left(e^{t(-\ln a+(x\ln c/\alpha)+2\ln a)}\right.\\
&\;\;\;\;\;\;\;\;\left.-\lambda e^{t(\ln b+(x\ln c/\alpha)+2\ln a)}\right)^{\alpha}.
\end{align*}
Applying the Binomial Theorem yields
\begin{align*}
&\sum_{n=0}^{\infty} \mathcal{G}_{n}^{(k,\alpha)}(x;\lambda,a,b,c) \frac{t^n}{n!}\\
&=\left(\frac{Li_{k}(1-e^{-2t(\ln ab)})}{-(\lambda e^{2t\ln (ab)}-1)}\right)^{\alpha}\sum_{j=0}^{\alpha}\binom{\alpha}{j}(-\lambda)^{\alpha-j}e^{jt(\ln a+(x\ln c/\alpha))}e^{(\alpha-j)t(\ln b+(x\ln c/\alpha)+2\ln a)}\\
&=\sum_{j=0}^{\alpha}\binom{\alpha}{j}(-1)^j\lambda^{\alpha-j}\left(\frac{Li_{k}(1-e^{-2t(\ln ab)})}{\lambda^2 e^{2t\ln (ab)}-1}\right)^{\alpha}e^{[((\alpha-j)\ln b+x\ln c+(2\alpha-j)\ln a)/2\ln ab](2t\ln ab)}.
\end{align*}
Using the definition of Apostol-type poly-Bernoulli polynomials in \eqref{ApostolPolyBernoulli}, we have
\begin{align*}
&\sum_{n=0}^{\infty} \mathcal{G}_{n}^{(k)}(x;\lambda,a,b,c) \frac{t^n}{n!}\\
&=\sum_{j=0}^{\alpha}\binom{\alpha}{j}(-1)^j\lambda^{\alpha-j}\left\{\sum_{n=0}^{\infty} \mathcal{B}_{n}^{(k,\alpha)} \left( \frac {(\alpha-j)\ln b+x\ln c+(2\alpha-j)\ln a}{2(\ln a+\ln b)};\lambda^2 \right) 2^{n}(\ln ab)^{n}\frac{t^n}{n!}\right\}\\
&=\sum_{n=0}^{\infty}\left\{\sum_{j=0}^{\alpha}\binom{\alpha}{j}(-1)^j\lambda^{\alpha-j} \mathcal{B}_{n}^{(k,\alpha)} \left( \frac {(\alpha-j)\ln b+x\ln c+(2\alpha-j)\ln a}{2(\ln a+\ln b)};\lambda^2 \right) 2^{n}(\ln ab)^{n}\right\}\frac{t^n}{n!}. 
\end{align*}
Comparing the coefficients of $\frac{t^n}{n!}$ yields \eqref{eq55}.
\end{proof}

\smallskip
The next theorem contains an identity that relates the Apostol-type poly-Genocchi polynomials of higher order with parameters $a, b$ and $c$ to Stirling numbers of the second kind $\rsts{n}{m}$ defined in \cite{ref22} by
\begin{equation}\label{eq14}
\sum_{n=m}^{\infty}\rsts{n}{m}\frac{t^n}{n!}=\frac{(e^t-1)^m}{m!}.
\end{equation}
Here, it is important to note that if $(c_0, c_1, \ldots, c_j, \ldots)$ is any sequence of numbers and $l$ is a positive integer, then
\begin{align}
\left(\sum_{j=0}^{\infty}c_j\frac{t^j}{j!}\right)^{l}&=\prod_{i=1}^{l}\left(\sum_{n_i=0}^{\infty}\frac{c_{n_i}}{n_i!}t^{n_i}\right)\nonumber\\
&=\sum_{n=0}^{\infty}\left\{\sum_{n_1+n_2+\ldots+n_{\alpha}=n}\prod_{i=1}^{l}c_{n_i}\binom{n}{n_1,n_2,\ldots,n_{\alpha}}\right\}\frac{t^n}{n!}.\label{series}
\end{align}
(see \cite{ref22}).

\bigskip
\begin{thrm}\label{type1ApostolPG-1}
The Apostol-type poly-Genocchi polynomials of higher order with parameters $a, b$ and $c$ satisfies the relation,
\begin{equation}\label{ident2-1}
\mathcal{G}_{n}^{(k,\alpha)}(x; \lambda, a, b, c)=\sum_{j=0}^{n}\binom{n}{j}(\ln a + \ln b)^{n-j} \mathcal{G}_{n-j}^{(\alpha)}\left(\ \frac {x \ln c + \alpha\ln a}{\ln a + \ln b};\lambda  \right)d_j,
\end{equation}
where
\begin{equation*}
d_j=\sum_{n_1+n_2+\ldots+n_{\alpha}=j}\prod_{i=1}^{\alpha}c_{n_i}\binom{j}{n_1,n_2,\ldots,n_{\alpha}}\;\;\mbox{and}\;\;c_j=\sum_{m=0}^{j}(-1)^{m+1}\frac{(2\ln ab)^jm!\rsts{j+1}{m+1}}{(j+1)(m+1)^{k-1}}.
\end{equation*}
\end{thrm}
\begin{proof}
Now, \eqref{eq711} can be written as
\begin{align*}
\sum_{n=0}^{\infty}\mathcal{G}_{n}^{(k,\alpha)}(x; \lambda, a, b, c)\frac{t^n}{n!}&=\frac{c^{xt}}{(a^{-t}+\lambda b^t)^{\alpha}}\left(\sum_{m=1}^{\infty}\frac{(1-e^{2t\ln ab})^m}{m^k}\right)^{\alpha}\\
&=\frac{c^{xt}}{(a^{-t}+\lambda b^t)^{\alpha}}\left(\sum_{m=0}^{\infty}\frac{(1-e^{2t\ln ab})^{m+1}}{(m+1)^k}\right)^{\alpha}\\
&=\frac{c^{xt}}{(a^{-t}+\lambda b^t)^{\alpha}}\left(\sum_{m=0}^{\infty}\frac{m!}{(m+1)^{k-1}}\frac{(1-e^{2t\ln ab})^{m+1}}{(m+1)!}\right)^{\alpha}\\
&=\frac{c^{xt}}{(a^{-t}+\lambda b^t)^{\alpha}}\left(\sum_{m=0}^{\infty}\frac{(-1)^{m+1}m!}{(m+1)^{k-1}}\sum_{j=m+1}^{\infty}\rsts{j}{m+1}\frac{(2t\ln ab)^j}{j!}\right)^{\alpha}\\
&=c^{xt}\left(\frac{2t\ln ab }{a^{-t}+\lambda b^t}\right)^{\alpha}\left(\sum_{m=0}^{\infty}\sum_{j=m}^{\infty}\frac{(-1)^{m+1}m!\rsts{j+1}{m+1}}{(j+1)(m+1)^{k-1}}\frac{(2t\ln ab)^j}{j!}\right)^{\alpha}.
\end{align*}
Using \eqref{eqq711}, we get
\begin{align*}
\sum_{n=0}^{\infty}\mathcal{G}_{n}^{(k,\alpha)}(x; \lambda, a, b, c)\frac{t^n}{n!}=\left(\sum_{n=0}^{\infty}\mathcal{G}_{n}^{(\alpha)}(x; \lambda, a, b, c)\frac{t^n}{n!}\right)\left(\sum_{j=0}^{\infty}c_j\frac{t^j}{j!}\right)^{\alpha},
\end{align*}
where 
$$c_j=\sum_{m=0}^{j}(-1)^{m+1}\frac{(2\ln ab)^jm!\rsts{j+1}{m+1}}{(j+1)(m+1)^{k-1}}.$$
Note that, using \eqref{series}, $\left(\sum_{j=0}^{\infty}c_j\frac{t^j}{j!}\right)^{\alpha}$ can be expressed as
\begin{align*}
\left(\sum_{j=0}^{\infty}c_j\frac{t^j}{j!}\right)^{\alpha}=\sum_{n=0}^{\infty}d_n\frac{t^n}{n!},
\end{align*}
where 
$$d_n=\sum_{n_1+n_2+\ldots+n_{\alpha}=n}\prod_{i=1}^{\alpha}c_{n_i}\binom{n}{n_1,n_2,\ldots,n_{\alpha}}.$$
It follows that
\begin{align*}
\sum_{n=0}^{\infty}\mathcal{G}_{n}^{(k,\alpha)}(x; \lambda, a, b, c)\frac{t^n}{n!}=\sum_{n=0}^{\infty}\left\{\sum_{j=0}^{n}\binom{n}{j}\mathcal{G}_{n-j}^{(\alpha)}(x; \lambda, a, b, c)d_j\right\}\frac{t^n}{n!}.
\end{align*}
Comparing the coefficients and using equation \eqref{eq10} complete the proof of the theorem.
\end{proof}

\begin{rema}\rm When $\alpha=1$, $d_j=c_j$.
\end{rema}

\smallskip
The identities in the following theorem are derived using the fact that the Apostol-type poly-Genocchi polynomials of higher order $\mathcal{G}_n^{(k,\alpha)}(x;\lambda, a, b)$ with parameters $a$ and $b$ satisfy the relation in \eqref{eq13}.

\begin{thrm}\label{thm38} The Apostol-type poly-Genocchi polynomials of higher order $\mathcal{G}_n^{(k,\alpha)}(x;\lambda, a, b, c)$ with parameters $a, b, c$ satisfy the following explicit formulas:

	\begin{equation} \label{eq115}
			 \mathcal{G}_n^{(k,\alpha)}(x;\lambda, a, b, c)=\sum_{m=0}^{\infty} \sum_{l=m}^{n} \begin{Bmatrix}l\\m\end{Bmatrix} \binom{n}{l}(\ln c)^l\mathcal{G}_{n-l}^{(k,\alpha)}(-m\ln c;\lambda,a,b)(x)^{(m)}
	\end{equation}
	\begin{equation}\label{eq116}
			\mathcal{G}_n^{(k,\alpha)}(x;\lambda, a, b, c)=\sum_{m=0}^{\infty} \sum_{l=m}^{n} \begin{Bmatrix}l\\m\end{Bmatrix} \binom{n}{l}(\ln c)^l \mathcal{G}_{n-l}^{(k,\alpha)}(\lambda,a,b)(x)_{m}
	\end{equation}
	\begin{equation}\label{eq117}
		\mathcal{G}_n^{(k,\alpha)}(x;\lambda, a, b)\sum_{l=0}^{n}\sum_{m=0}^{n-l}\binom{n}{l}\rsts{l+s}{s}\frac{\binom{n-l}{m}}{\binom{l+s}{s}}\mathcal{G}_{n-l-m}^{(k,\alpha)}(\lambda,a,b)B_m^{(s)}(x\ln c;\lambda)
	\end{equation}
	\begin{equation}\label{eq118}
		\mathcal{G}_n^{(k,\alpha)}(x;\lambda, a, b)=\sum_{m=0}^{\infty} \dfrac{\binom{n}{m}}{(1-\mu)^{s}} \sum_{j=0}^{s} {\binom{s}{j}} (-\mu)^{s-j}  \mathcal{G}_{n-m}^{(k,\alpha)}(j;\lambda,a,b)F_{m}^{(s)}(x;\mu),
	\end{equation}

\noindent where $(x)^{(n)}=x(x+1)\cdots(x+n-1),(x)_n=x(x-1)\cdots(x-n+1)$, the rising and falling factorials of $x$ of degree $n$ and $F_{n}^{(s)}(x;\mu)$, the Frobenius polynomials of higher order \cite{ref54-1}, defined by
$$\left(\ \dfrac{1-\mu}{e^{t}-\mu} \right)^{s}\ e^{xt}=\sum_{n=0}^{\infty}F_{n}^{(s)}(x;\mu)\frac{t^n}{n!}.$$
\end{thrm} 

\begin{proof} Proving relations \eqref{eq115}-\eqref{eq117} makes use of the definition of Stirling numbers of the second kind in \eqref{eq14}. To prove \eqref{eq115}, using the generalized Binomial Theorem, \eqref{eq711} may be written as
$$\sum_{n=0}^{\infty}\mathcal{G}_n^{(k,\alpha)}(x;\lambda,a,b,c)\frac{t^n}{n!}=\left(\frac{Li_{k}(1-(ab)^{-2t})}{a^{-t}+\lambda b^t}\right)^{\alpha}\sum_{m=0}^{\infty}\binom{x+m-1}{m} (1-e^{-t\ln c})^{m}$$
\begin{align*}
&=\sum_{m=0}^{\infty}(x)^{(m)} \frac{(e^{t\ln c}-1)^{m}}{m!}\left(\frac{Li_{k}(1-(ab)^{-2t})}{a^{-t}+\lambda b^t}\right)^{\alpha}e^{-mt\ln c}\\
&=\sum_{m=0}^{\infty}(x)^{(m)} \left(\ \sum_{n=0}^{\infty} \begin{Bmatrix}n\\m\end{Bmatrix}\frac{(t\ln c)^n}{n!}\right)\ \left(\ \sum_{n=0}^{\infty} \mathcal{G}_{n}^{(k,\alpha)}(-m\ln c;\lambda,a,b)\frac{t^n}{n!}\right)\ \\
&=\sum_{m=0}^{\infty}(x)^{(m)}\sum_{n=0}^{\infty}\sum_{l=0}^{n}\begin{Bmatrix}l\\m\end{Bmatrix}(\ln c)^l\frac{t^{l}}{l!} \mathcal{G}_{n-l}^{(k,\alpha)}(-m\ln c;\lambda,a,b)\frac{t^{n-l}}{(n-l)!}\\
&=\sum_{n=0}^{\infty} \bigg\{ \sum_{m=0}^{\infty} \sum_{l=m}^{n} \begin{Bmatrix}l\\m\end{Bmatrix} \binom{n}{l}(\ln c)^l\mathcal{G}_{n-l}^{(k,\alpha)}(-m\ln c;\lambda,a,b)(x)^{(m)} \bigg\} \frac{t^{n}}{n!}.  
\end{align*}

\noindent Comparing coefficients completes the proof of \eqref{eq115}. To prove identity \eqref{eq116}, \eqref{eq711} may be written as 
\begin{align*}
		\sum_{n=0}^{\infty}\mathcal{G}_n^{(k,\alpha)}(x;\lambda,a,b,c)\frac{t^n}{n!}&=\left(\frac{Li_{k}(1-(ab)^{-2t})}{a^{-t}+\lambda b^t}\right)^{\alpha}\sum_{m=0}^{\infty}\binom{x}{m} (e^{t\ln c}-1)^{m}\\
&=\sum_{m=0}^{\infty}(x)_{m} \frac{(e^{t\ln c}-1)^{m}}{m!}\left(\frac{Li_{k}(1-(ab)^{-2t})}{a^{-t}+\lambda b^t}\right)^{\alpha}\\
&=\sum_{m=0}^{\infty}(x)_{m} \left(\ \sum_{n=0}^{\infty} \begin{Bmatrix}n\\m\end{Bmatrix}\frac{(t\ln c)^n}{n!}\right)\ \left(\ \sum_{n=0}^{\infty} \mathcal{G}_{n}^{(k,\alpha)}(0;\lambda,a,b)\frac{t^n}{n!}\right)\ \\
&=\sum_{m=0}^{\infty}(x)_{m}\sum_{n=0}^{\infty}\sum_{l=0}^{n}\begin{Bmatrix}l\\m\end{Bmatrix}(\ln c)^l\frac{t^{l}}{l!} \mathcal{G}_{n-l}^{(k,\alpha)}(\lambda,a,b)\frac{t^{n-l}}{(n-l)!}\\
&=\sum_{n=0}^{\infty} \bigg\{ \sum_{m=0}^{\infty} \sum_{l=m}^{n} \begin{Bmatrix}l\\m\end{Bmatrix} \binom{n}{l}(\ln c)^l\mathcal{G}_{n-l}^{(k,\alpha)}(\lambda,a,b)(x)_{m} \bigg\} \frac{t^{n}}{n!}.  
\end{align*}

\noindent Again, comparing coefficients completes the proof of \eqref{eq116}. To prove relation \eqref{eq117}, using \eqref{ApostolBernoulli}, \eqref{eq711} may be expressed as
$$\sum_{n=0}^{\infty}\mathcal{G}_n^{(k,\alpha)}(x;\lambda,a,b,c)\frac{t^n}{n!}=\left(\frac{(e^t-1)^s}{s!}\right)\left(\frac{t^se^{xt\ln c}}{(e^t-1)^s}\right)\left(\frac{Li_{k}(1-(ab)^{-2t})}{a^{-t}+\lambda b^t}\right)^{\alpha}\frac{s!}{t^s}\qquad\qquad\qquad$$
\begin{align*}
&=\left(\sum_{n=0}^{\infty}\rsts{n+s}{s}\frac{t^{n+s}}{(n+s)!}\right)\left(\sum_{m=0}^{\infty}B_m^{(s)}(x\ln c;\lambda)\frac{t^{m}}{m!}\right)\left(\sum_{n=0}^{\infty}\mathcal{G}_n^{(k,\alpha)}(0;\lambda,a,b)\frac{t^{m}}{m!}\right)\frac{s!}{t^s}\\
&=\left(\sum_{n=0}^{\infty}\rsts{n+s}{s}\frac{t^{n+s}}{(n+s)!}\right)\left(\sum_{n=0}^{\infty}\sum_{m=0}^{n}\binom{n}{m}B_m^{(s)}(x\ln c;\lambda)\mathcal{G}_{n-m}^{(k,\alpha)}(\lambda,a,b)\frac{t^{n}}{n!}\right)\frac{s!}{t^s}\\
&=\left(\sum_{n=0}^{\infty}\sum_{l=0}^{n}\rsts{l+s}{s}\frac{t^{l+s}}{(l+s)!}\sum_{m=0}^{n-l}\binom{n-l}{m}B_m^{(s)}(x\ln c;\lambda)\mathcal{G}_{n-l-m}^{(k,\alpha)}(\lambda,a,b)\frac{t^{n-l}}{(n-l)!}\right)\frac{s!}{t^s}.
\end{align*}
This can further be written as
\begin{align*}
&\sum_{n=0}^{\infty}\mathcal{G}_n^{(k,\alpha)}(x;\lambda,a,b,c)\frac{t^n}{n!}\\
&=\left(\sum_{l=0}^{\infty}\sum_{n=l}^{\infty}\sum_{m=0}^{n-l}\rsts{l+s}{s}\frac{l!s!}{(l+s)!}\binom{n-l}{m}B_m^{(s)}(x\ln c;\lambda)\mathcal{G}_{n-l-m}^{(k,\alpha)}(\lambda,a,b)\frac{n!}{(n-l)!l!}\frac{t^{n}}{n!}\right)\\
&=\sum_{n=0}^{\infty}\left(\sum_{l=0}^{n}\sum_{m=0}^{n-l}\binom{n}{l}\rsts{l+s}{s}\frac{\binom{n-l}{m}}{\binom{l+s}{s}}B_m^{(s)}(x\ln c;\lambda)\mathcal{G}_{n-l-m}^{(k,\alpha)}(\lambda,a,b)\right)\frac{t^{n}}{n!}.
\end{align*}
Comparing the coefficients of $\frac{t^{n}}{n!}$ gives \eqref{eq117}. To prove relation \eqref{eq118}, \eqref{eq711} may be expressed as
$$\sum_{n=0}^{\infty}\mathcal{G}_n^{(k,\alpha)}(x;\lambda,a,b,c)\frac{t^n}{n!}=\left(\frac{(1-\mu)^s}{(e^t-\mu)^s}e^{xt\ln c}\right)\left(\frac{(e^t-\mu)^s}{(1-\mu)^s}\right)\left(\frac{Li_{k}(1-(ab)^{-2t})}{a^{-t}+\lambda b^t}\right)^{\alpha}\qquad\qquad\qquad$$
\begin{align*}
&=\frac{1}{(1-\mu)^s}\left(\sum_{n=0}^{\infty}F_n^{(s)}(x\ln c;\mu)\frac{t^{n}}{n!}\right)\left(\sum_{j=0}^{s}\binom{s}{j}(-\mu)^{s-j}\left(\frac{Li_{k}(1-(ab)^{-2t})}{a^{-t}+\lambda b^t}\right)^{\alpha}e^{jt}\right)\\
&=\frac{1}{(1-\mu)^s}\sum_{j=0}^{s}\binom{s}{j}(-\mu)^{s-j}\left(\sum_{n=0}^{\infty}F_n^{(s)}(x\ln c;\mu)\frac{t^{n}}{n!}\right)\left(
\sum_{n=0}^{\infty}\mathcal{G}_n^{(k,\alpha)}(x;\lambda,a,b)\frac{t^n}{n!}\right)\\
&=\frac{1}{(1-\mu)^s}\sum_{j=0}^{s}\binom{s}{j}(-\mu)^{s-j}\sum_{n=0}^{\infty}\left(\sum_{m=0}^{n}\binom{n}{m}\mathcal{G}_{n-m}^{(k,\alpha)}(x;\lambda,a,b)F_m^{(s)}(x\ln c;\mu)\right)\frac{t^{n}}{n!}\\
&=\sum_{n=0}^{\infty}\left(\sum_{m=0}^{n}\frac{\binom{n}{m}}{(1-\mu)^s}\sum_{j=0}^{s}\binom{s}{j}(-\mu)^{s-j}\mathcal{G}_{n-m}^{(k,\alpha)}(x;\lambda,a,b)F_m^{(s)}(x\ln c;\mu)\right)\frac{t^{n}}{n!}.
\end{align*}
Comparing the coefficients of $\frac{t^{n}}{n!}$ gives \eqref{eq118}.
\end{proof}

\section{Symmetrized Generalization}

\begin{deftn}\label{defn5}
For $m, n\ge0$, we define the symmetrized generalization of multi poly-Genocchi polynomials with parameters $a$, $b$ and $c$ as follows,
\begin{equation}\label{symgen}
\mathcal{S}^{(m,\alpha)}_n(x,y;\lambda,a,b,c)=\sum_{k=0}^m\binom{m}{k}\frac{\mathcal{G}^{(-k,\alpha)}_n(x;\lambda,a,b,c)}{(\ln a+\ln b)^n}\left(\frac{y\ln c+\alpha\ln a}{\ln a +\ln b}\right)^{m-k}.
\end{equation}
\end{deftn}

\smallskip
The following theorem contains the double generating function for $\mathcal{S}^{(m)}_n(x,y;a,b,c)$.
\begin{thrm}\label{thmm4}
For $n,m\ge0$, we have
\begin{equation}\label{eqnnnew6}
\sum_{n=0}^{\infty}\sum_{m=0}^{\infty}\mathcal{S}^{(m,\alpha)}_n(x,y;\lambda,a,b,c)\frac{t^n}{n!}\frac{u^m}{m!}=\frac{e^{\left(\frac{y\ln c+\alpha\ln a}{\ln a +\ln b}\right)u}e^{\left(\frac{x \ln c + \alpha\ln a}{\ln a + \ln b}\right)t}e^{2t}}{(1+\lambda e^{t})(e^{2t}-e^{2t+u}+e^u)}.
\end{equation}
\begin{proof}
$$\sum_{n=0}^{\infty}\sum_{m=0}^{\infty}\mathcal{S}^{(m,\alpha)}_n(x,y;\lambda,a,b,c)\frac{t^n}{n!}\frac{u^m}{m!}\qquad\qquad\qquad\qquad\qquad\qquad\qquad\qquad\qquad\qquad\qquad\qquad\qquad\qquad\qquad\qquad$$
\begin{eqnarray*}
&=&\sum_{n=0}^{\infty}\sum_{m=0}^{\infty}\sum_{k=0}^m\frac{\mathcal{G}^{(-k,\alpha)}_n(x;\lambda,a,b,c)}{(\ln a+\ln b)^n}\left(\frac{y\ln c+\alpha\ln a}{\ln a +\ln b}\right)^{m-k}\frac{t^n}{n!}\frac{u^m}{k!(m-k)!}\\
&=&\sum_{n=0}^{\infty}\sum_{k=0}^{\infty}\sum_{m=k}^{\infty}\frac{\mathcal{G}^{(-k,\alpha)}_n(x;\lambda,a,b,c)}{(\ln a+\ln b)^n}\left(\frac{y\ln c+\alpha\ln a}{\ln a +\ln b}\right)^{m-k}\frac{t^n}{n!}\frac{u^m}{k!(m-k)!}\\
&=&\sum_{n=0}^{\infty}\sum_{k=0}^{\infty}\frac{\mathcal{G}^{(-k,\alpha)}_n(x;\lambda,a,b,c)}{(\ln a+\ln b)^n}\frac{t^n}{n!}\frac{u^k}{k!}\sum_{l=0}^{\infty}\left(\frac{y\ln c+\alpha\ln a}{\ln a +\ln b}\right)^{l}\frac{u^l}{l!}\\
&=&e^{\left(\frac{y\ln c+\alpha\ln a}{\ln a +\ln b}\right)u}\sum_{k=0}^{\infty}\sum_{n=0}^{\infty}\frac{\mathcal{G}^{(-k,\alpha)}_n(\lambda,a,b,c)}{(\ln a+\ln b)^n}\frac{t^n}{n!}\frac{u^k}{k!}.
\end{eqnarray*}
Applying \eqref{eq10} yields
\begin{eqnarray*}
\sum_{n=0}^{\infty}\sum_{m=0}^{\infty}\mathcal{S}^{(m,\alpha)}_n(x,y;\lambda,a,b,c)\frac{t^n}{n!}\frac{u^m}{m!}&=&e^{\left(\frac{y\ln c+\alpha\ln a}{\ln a +\ln b}\right)u}\sum_{k=0}^{\infty}\sum_{n=0}^{\infty}\mathcal{G}_n^{(-k,\alpha)} \left(\ \frac{x \ln c + \alpha\ln a}{\ln a + \ln b};\lambda \right)\frac{t^n}{n!}\frac{u^k}{k!}.
\end{eqnarray*}
Now, using \eqref{ApostolGenoPoly}, we obtain
\begin{eqnarray*}
\sum_{n=0}^{\infty}\sum_{m=0}^{\infty}\mathcal{S}^{(m,\alpha)}_n(x,y;\lambda,a,b,c)\frac{t^n}{n!}\frac{u^m}{m!}
&=&e^{\left(\frac{y\ln c+\alpha\ln a}{\ln a +\ln b}\right)u}e^{\left(\frac{x \ln c + \alpha\ln a}{\ln a + \ln b}\right)t}\sum_{k=0}^{\infty}\frac{L_{i_{(-k)}}(1-e^{-2t})}{1+\lambda e^{t}}\frac{u^k}{k!}\\
&=&\frac{e^{\left(\frac{y\ln c+\alpha\ln a}{\ln a +\ln b}\right)u}e^{\left(\frac{x \ln c + \alpha\ln a}{\ln a + \ln b}\right)t}}{1+\lambda e^{t}}\sum_{k=0}^{\infty}L_{i_{(-k)}}(1-e^{-2t})\frac{u^k}{k!}.
\end{eqnarray*}
Employing the definition of polylogarithm yields
\begin{eqnarray*}
\sum_{n=0}^{\infty}\sum_{m=0}^{\infty}\mathcal{S}^{(m,\alpha)}_n(x,y;\lambda,a,b,c)\frac{t^n}{n!}\frac{u^m}{m!}&=&\frac{e^{\left(\frac{y\ln c+\alpha\ln a}{\ln a +\ln b}\right)u}e^{\left(\frac{x \ln c + \alpha\ln a}{\ln a + \ln b}\right)t}}{1+\lambda e^{t}}\sum_{k=0}^{\infty}\sum_{m=0}^{\infty}\frac{\left(1-e^{-2t}\right)^m}{m^{-k}}\frac{u^k}{k!}\\
&=&\frac{e^{\left(\frac{y\ln c+\alpha\ln a}{\ln a +\ln b}\right)u}e^{\left(\frac{x \ln c + \alpha\ln a}{\ln a + \ln b}\right)t}}{(1+\lambda e^{t})(1-((1-e^{-2t})e^u))}\\
&=&\frac{e^{\left(\frac{y\ln c+\alpha\ln a}{\ln a +\ln b}\right)u}e^{\left(\frac{x \ln c + \alpha\ln a}{\ln a + \ln b}\right)t}e^{2t}}{(1+\lambda e^{t})(e^{2t}-e^{2t+u}+e^u)}.
\end{eqnarray*}
\end{proof}
\end{thrm}

\smallskip
The Apostol-type poly-Genocchi polynomials discussed above will be referred to as type 1 Apostol-poly-Genocchi polynomials. Type 2 of these polynomials are introduced in the next section.

\section{Type 2 Higher Order Apostol-Poly-Genocchi Polynomials}
In this section, we will consider another variation of Genocchi polynomials using the concept of polyexponential function \cite{ref45-1} defined by
\begin{equation}\label{polyexpo}
e_k(z)=\sum_{m=1}^{\infty}\frac{z^m}{(m-1)!m^k}.
\end{equation}
Note that when $k=1$, $e_1(z)=e^z-1$. Hence, if $z=\log (1+2t)$, 
$$e_1(z)=e_1(\log (1+2t))=e^{\log (1+2t)}-1=2t.$$

\begin{deftn}\rm The type 2 Apostol-poly-Genocchi polynomials of higher order with parameters $a, b$ and $c$, denoted by $\mathcal{G}_{n,2}^{(k)}(x; \lambda, a, b, c)$, are defined as follows,
\begin{equation}\label{eq811-1}
\sum_{n=0}^{\infty}\mathcal{G}_{n,2}^{(k,\alpha)}(x; \lambda, a, b, c)\frac{t^n}{n!}=\left(\frac{e_{k}(\log(1+2t\ln ab))}{a^{-t}+\lambda b^t}\right)^{\alpha}c^{xt},\;\;\; |t| < \frac{\sqrt{(\ln \lambda)^2+\pi^2}}{|\ln a + \ln b|}.
\end{equation}
\end{deftn}

\bigskip
The following are some special cases of $\mathcal{G}_{n,2}^{(k,\alpha)}(x; \lambda, a, b, c)$:
\begin{enumerate}
\item When $x=0$, we use $\mathcal{G}_{n,2}^{(k,\alpha)}(\lambda, a, b)$ to denote $\mathcal{G}_{n,2}^{(k,\alpha)}(0; \lambda, a, b, c)$, the type 2 Apostol-poly-Genocchi numbers with parameters $a, b$. That is,
\begin{equation}\label{eq811-0}
\sum_{n=0}^{\infty}\mathcal{G}_{n,2}^{(k,\alpha)}(\lambda, a, b)\frac{t^n}{n!}=\left(\frac{e_{k}(\log(1+2t\ln ab))}{a^{-t}+\lambda b^t}\right)^{\alpha}.
\end{equation}
\item When $a=1, b=c=e$, \eqref{eq811-1} yields
\begin{equation}\label{eq811-00}
\sum_{n=0}^{\infty}\mathcal{G}_{n,2}^{(k,\alpha)}(x; \lambda)\frac{t^n}{n!}=\left(\frac{e_{k}(\log(1+2t))}{1+\lambda e^t}\right)^{\alpha}e^{xt},
\end{equation}
where the polynomials $\mathcal{G}_{n,2}^{(k,\alpha)}(x; \lambda)=\mathcal{G}_{n,2}^{(k,\alpha)}(x; \lambda, 1, e, e)$ are called the type 2 Apostol-poly-Genocchi polynomials.
\item When $k=1$, \eqref{eq811-1} gives
\begin{equation}\label{eq811-2}
\sum_{n=0}^{\infty}\mathcal{G}_{n,2}^{(\alpha)}(x; \lambda, a, b, c)\frac{t^n}{n!}=\left(\frac{2t\ln ab}{a^{-t}+\lambda b^t}\right)^{\alpha}c^{xt},
\end{equation}
where the polynomials $\mathcal{G}_{n,2}^{(\alpha)}(x; \lambda, a, b, c)=\mathcal{G}_{n,2}^{(1,\alpha)}(x; \lambda, a, b, c)$ are called the type 2 Apostol-Genocchi polynomials with parameters $a, b$ and $c$, which are related to the type 1 Apostol-Genocchi polynomials with parameters $a, b$ and $c$ as follows
$$\mathcal{G}_{n,2}^{(1,\alpha)}(x; \lambda, a, b, c)=\frac{\mathcal{G}_{n}^{(1,\alpha)}(x; \lambda, a, b, c)}{\ln ab}.$$
\item When $a=1, b=c=e$, \eqref{eq811-2} yields
\begin{equation}\label{eq811-3}
\sum_{n=0}^{\infty}\mathcal{G}_{n,2}^{(1,\alpha)}(x; \lambda, 1, e, e)\frac{t^n}{n!}=\left(\frac{2t}{1+\lambda e^t}\right)^{\alpha}e^{xt},
\end{equation}
where the polynomials $\mathcal{G}_{n,2}^{(1,\alpha)}(x; \lambda, 1, e, e)=\mathcal{G}_{n}^{(\alpha)}(x; \lambda)$ are the type 2 Apostol-Genocchi polynomials in \eqref{eqn2}.
Furthermore, when $\alpha=1$,
\begin{equation}\label{eq811-33}
\sum_{n=0}^{\infty}\mathcal{G}_{n,2}^{(1)}(x; \lambda)\frac{t^n}{n!}=\frac{2t}{1+\lambda e^t}e^{xt},
\end{equation}
where $\mathcal{G}_{n,2}(x; \lambda)=\mathcal{G}_{n,2}^{(1)}(x; \lambda)$, the type 2 Apostol-Genocchi polynomials.
\end{enumerate}

\smallskip
Rewrite \eqref{eq811-1} as follows
\begin{align*}
\sum_{n=0}^{\infty} \mathcal{G}_{n,2}^{(k,\alpha)}(x;\lambda, a, b, c)\frac{t^n}{n!} &=\left(\frac{e_{k}(\log(1+2t\ln ab))}{a^{-t}(1+\lambda (ab)^t)}\right)^{\alpha}e^{xt \ln c} \\   
&= e^{\frac{x \ln c +\alpha\ln a}{\ln ab}t \ln ab} \left(\frac{e_{k}(\log(1+2t\ln ab))}{1+\lambda e^{t \ln ab}}\right)  \\ 
&=\sum_{n=0}^{\infty} (\ln a + \ln b)^{n} \mathcal{G}_{n,2}^{(k,\alpha)} \left(\ \frac {x \ln c + \alpha\ln a}{\ln a + \ln b};\lambda  \right)\ \frac {t^n}{n!}, 
\end{align*}
 and comparing the coefficients yield the following theorem.

\begin{thrm}\label{type2ApostolPG}
The type 2 Apostol-poly-Genocchi polynomials with parameters $a, b$ and $c$ satisfies the relation,
\begin{equation}\label{ident0}
\mathcal{G}_{n,2}^{(k,\alpha)}(x;\lambda, a, b, c)=(\ln a + \ln b)^{n} \mathcal{G}_{n,2}^{(k,\alpha)}\left(\ \frac {x \ln c + \alpha\ln a}{\ln a + \ln b};\lambda  \right).
\end{equation}
\end{thrm}

\smallskip
When $k=1$, \eqref{ident0} reduces to the following relation
\begin{equation}\label{ident1}
\mathcal{G}_{n,2}^{(\alpha)}(x;\lambda, a, b, c)=(\ln a + \ln b)^{n-j} \mathcal{G}_{n,2}^{(\alpha)}\left(\ \frac {x \ln c + \alpha\ln a}{\ln a + \ln b};\lambda  \right).
\end{equation}

The next theorem contains an identity that relates the type 2 Apostol-poly-Genocchi polynomials of higher order with parameters $a, b$ and $c$ to Stirling numbers of the first kind $\rstf{n}{m}$ defined by
\begin{equation}\label{SNFK}
\sum_{n=m}^{\infty}\rstf{n}{m}\frac{t^n}{n!}=\frac{(\log(1+t))^m}{m!}.
\end{equation}

\begin{thrm}\label{type2ApostolPG-1}
The type 2 Apostol-poly-Genocchi polynomials of higher order with parameters $a, b$ and $c$ satisfies the relation,
\begin{equation}\label{ident2}
\mathcal{G}_{n,2}^{(k,\alpha)}(x; \lambda, a, b, c)=\sum_{j=0}^{n}\binom{n}{j}(\ln a + \ln b)^{n-j} \mathcal{G}_{n-j,2}^{(\alpha)}\left(\ \frac {x \ln c + \alpha\ln a}{\ln a + \ln b};\lambda  \right)d_j,
\end{equation}
where
\begin{equation*}
d_j=\sum_{n_1+n_2+\ldots+n_{\alpha}=j}\prod_{i=1}^{\alpha}c_{n_i}\binom{j}{n_1,n_2,\ldots,n_{\alpha}}\;\;\mbox{and}\;\;c_j=\sum_{m=0}^{j}\frac{(2\ln ab)^j\rstf{j+1}{m+1}}{(j+1)(m+1)^{k-1}}.
\end{equation*}
\end{thrm}
\begin{proof}
Applying the definition of polyexponential function \eqref{polyexpo}, \eqref{eq811-1} may be written as
\begin{align*}
\sum_{n=0}^{\infty}\mathcal{G}_{n,2}^{(k,\alpha)}(x; \lambda, a, b, c)\frac{t^n}{n!}&=\frac{c^{xt}}{(a^{-t}+\lambda b^t)^{\alpha}}\left(\sum_{m=1}^{\infty}\frac{(\log(1+2t\ln ab))^m}{(m-1)!m^k}\right)^{\alpha}\\
&=\frac{c^{xt}}{(a^{-t}+\lambda b^t)^{\alpha}}\left(\sum_{m=0}^{\infty}\frac{(\log(1+2t\ln ab))^{m+1}}{m!(m+1)^k}\right)^{\alpha}\\
&=\frac{c^{xt}}{(a^{-t}+\lambda b^t)^{\alpha}}\left(\sum_{m=0}^{\infty}\frac{1}{(m+1)^{k-1}}\frac{\log(1+2t\ln ab))^{m+1}}{(m+1)!}\right)^{\alpha}.
\end{align*}
This can further be written, using \eqref{SNFK}, as follows
\begin{align*}
\sum_{n=0}^{\infty}\mathcal{G}_{n,2}^{(k,\alpha)}(x; \lambda, a, b, c)\frac{t^n}{n!}
&=\frac{c^{xt}}{(a^{-t}+\lambda b^t)^{\alpha}}\left(\sum_{m=0}^{\infty}\frac{1}{(m+1)^{k-1}}\sum_{j=m+1}^{\infty}\rstf{j}{m+1}\frac{(2t\ln ab)^j}{j!}\right)^{\alpha}\\
&=c^{xt}\left(\frac{2t\ln ab }{a^{-t}+\lambda b^t}\right)^{\alpha}\left(\sum_{m=0}^{\infty}\sum_{j=m}^{\infty}\frac{\rstf{j+1}{m+1}}{(j+1)(m+1)^{k-1}}\frac{(2t\ln ab)^j}{j!}\right)^{\alpha}.
\end{align*}
Applying \eqref{eq811-2} yields
\begin{align}
\sum_{n=0}^{\infty}\mathcal{G}_{n,2}^{(k,\alpha)}(x; \lambda, a, b, c)\frac{t^n}{n!}=\left(\sum_{n=0}^{\infty}\mathcal{G}_{n,2}^{(\alpha)}(x; \lambda, a, b, c)\frac{t^n}{n!}\right)\left(\sum_{j=0}^{\infty}c_j\frac{t^j}{j!}\right)^{\alpha},\label{eqqq}
\end{align}
where 
$$c_j=\sum_{m=0}^{j}\frac{(2\ln ab)^j\rstf{j+1}{m+1}}{(j+1)(m+1)^{k-1}}.$$
Note that, using \eqref{series}, equation \eqref{eqqq} can be expressed as
\begin{align*}
\sum_{n=0}^{\infty}\mathcal{G}_{n,2}^{(k,\alpha)}(x; \lambda, a, b, c)\frac{t^n}{n!}=\sum_{n=0}^{\infty}\left\{\sum_{j=0}^{n}\binom{n}{j}\mathcal{G}_{n-j,2}^{(\alpha)}(x; \lambda, a, b, c)d_j\right\}\frac{t^n}{n!},
\end{align*}
where 
$$d_j=\sum_{n_1+n_2+\ldots+n_{\alpha}=j}\prod_{i=1}^{\alpha}c_{n_i}\binom{j}{n_1,n_2,\ldots,n_{\alpha}}.$$
This immediately gives \eqref{ident2} by comparing the coefficients and using equation \eqref{ident0}.
\end{proof}

\smallskip
The next theorem shows the relationship between the type 2 Apostol-poly-Genocchi polynomials of higher order with parameters $a$, $b$ and $c$ and the type 2 Apostol-poly-Bernoulli polynomials, which may be defined as 
\begin{equation}\label{ApostolPolyBernoulli-o1}
\left(\frac{e_{k}(\log(1+t))}{\lambda e^{t}-1}\right)^{\alpha}e^{xt}=\sum_{n=0}^{\infty}\mathcal{B}_{n,2}^{(k,\alpha)}(x;\lambda)\frac{x^n}{n!}.
\end{equation}
These polynomials and those in \eqref{ApostolPolyBernoulli} are certain generalizations of Bernoulli-type polynomials and are also worthy for investigation.

\smallskip
\begin{thrm} The type 2 Apostol-poly-Genocchi polynomials of higher order with parameters $a, b, c$ satisfy the relation,
\begin{equation}\label{eq13-o11}
\mathcal{G}_{n,2}^{(k)}(x;\lambda,a,b,c)=\sum_{j=0}^{\alpha}\binom{\alpha}{j}(-1)^j\lambda^{\alpha-j}\mathcal{B}_{n,2}^{(k,\alpha)} \left( \frac {(\alpha-j)\ln b+x\ln c+(2\alpha-j)\ln a}{2(\ln a+\ln b)};\lambda^2 \right) 2^{n}(\ln ab)^{n}. 
\end{equation}
\end{thrm}

\begin{proof} Rewrite equation \eqref{eq811-1} as
\begin{align*}
\sum_{n=0}^{\infty} \mathcal{G}_{n,2}^{(k,\alpha)}(x;\lambda,a,b,c) \frac{t^n}{n!}&=\left(\frac{e_{k}(\log(1+2t\ln ab))}{a^{-t}+\lambda b^{t}}\right)^{\alpha} e^{xt \ln c}\\
&=\left(\frac{e_{k}(\log(1+2t\ln ab))}{a^{-2t}-(\lambda b^{t})^2}\right)^{\alpha}(a^{-t}-\lambda b^{t})^{\alpha} e^{xt \ln c}\\
&=\left(\frac{e_{k}(\log(1+2t\ln ab))}{(1-(\lambda (ab)^{t})^2)}\right)^{\alpha}(e^{-t\ln a }-\lambda e^{t \ln b})^{\alpha}e^{xt \ln c} e^{2t\alpha\ln a}\\
&=\left(\frac{e_{k}(\log(1+2t\ln ab))}{-(\lambda^2 e^{2t\ln (ab)}-1)}\right)^{\alpha}\left(e^{t(-\ln a+(x\ln c/\alpha)+2\ln a)}\right.\\
&\;\;\;\;\;\;\;\;\left.-\lambda e^{t(\ln b+(x\ln c/\alpha)+2\ln a)}\right)^{\alpha}.
\end{align*}
Applying the Binomial Theorem yields
\begin{align*}
&\sum_{n=0}^{\infty} \mathcal{G}_{n,2}^{(k,\alpha)}(x;\lambda,a,b,c) \frac{t^n}{n!}\\
&=\left(\frac{e_{k}(\log(1+2t\ln ab))}{-(\lambda^2 e^{2t\ln (ab)}-1)}\right)^{\alpha}\sum_{j=0}^{\alpha}\binom{\alpha}{j}(-\lambda)^{\alpha-j}e^{jt(\ln a+(x\ln c/\alpha))}e^{(\alpha-j)t(\ln b+(x\ln c/\alpha)+2\ln a)}\\
&=\sum_{j=0}^{\alpha}\binom{\alpha}{j}(-1)^j\lambda^{\alpha-j}\left(\frac{e_{k}(\log(1+2t\ln ab))}{\lambda^2 e^{2t\ln (ab)}-1}\right)^{\alpha}e^{[((\alpha-j)\ln b+\alpha(x\ln c/\alpha)+(2\alpha-j)\ln a)/2\ln ab](2t\ln ab)}.
\end{align*}
Using the definition of type 2 Apostol-poly-Bernoulli polynomials of higher order in \eqref{ApostolPolyBernoulli-o1}, we have
\begin{align*}
&\sum_{n=0}^{\infty} \mathcal{G}_{n,2}^{(k,\alpha)}(x;\lambda,a,b,c) \frac{t^n}{n!}\\
&=\sum_{j=0}^{\alpha}\binom{\alpha}{j}(-1)^j\lambda^{\alpha-j}\left\{\sum_{n=0}^{\infty} \mathcal{B}_{n,2}^{(k,\alpha)} \left( \frac {(\alpha-j)\ln b+x\ln c+(2\alpha-j)\ln a}{2(\ln a+\ln b)};\lambda^2 \right) 2^{n}(\ln ab)^{n}\frac{t^n}{n!}\right\}\\
&=\sum_{n=0}^{\infty}\left\{\sum_{j=0}^{\alpha}\binom{\alpha}{j}(-1)^j\lambda^{\alpha-j}\mathcal{B}_{n,2}^{(k,\alpha)} \left( \frac {(\alpha-j)\ln b+x\ln c+(2\alpha-j)\ln a}{2(\ln a+\ln b)};\lambda^2 \right) 2^{n}(\ln ab)^{n}\right\}\frac{t^n}{n!}.
\end{align*}
Comparing the coefficients of $\frac{t^n}{n!}$ yields \eqref{eq13-o11}.
\end{proof}

\smallskip
\begin{rema}\rm It is left to the reader to prove the following identities which can be done parallel to the proofs of the corresponding identities in sections 3, 4 and 5 for type 1 Apostol-poly-Genocchi polynomials of higher order with parameters $a$, $b$ and $c$:
\begin{align*}
 \mathcal{G}_{n,2}^{(k,\alpha)}(x;\lambda, a, b, c)&=\sum_{i=0}^{n} \binom {n}{i} (\ln c)^{n-i} \mathcal{G}_{i,2}^{(k,\alpha)}(\lambda,a,b) x^{n-i}\\
\mathcal{G}_{n,2}^{(k,\alpha)}(x+1;\lambda,a,b,c)&=\sum_{r=0}^n\binom{n}{r}(\ln c)^r\mathcal{G}_{n-r,2}^{(k,\alpha)}(x;\lambda,a,b,c)\\
 \frac {d}{dx} \mathcal{G}_{n+1,2}^{(k,\alpha)}(x;\lambda, a, b, c)&= (n+1)(\ln c) \mathcal{G}_{n,2}^{(k,\alpha)} (x;\lambda,a,b,c)\\
\mathcal{G}_{n,2}^{(k,\alpha)} (x+y;\lambda,a,b,c)&=\sum_{i=0}^{\infty} \binom{n}{i} (\ln c)^{n-i}\mathcal{G}_{i,2}^{(k,\alpha)} (x;\lambda,a,b,c) y^{n-i}
\end{align*}
\begin{align*}
\mathcal{G}_{n,2}^{(k,\alpha)}(x;\lambda, a, b, c)&=\sum_{m=0}^{\infty} \sum_{l=m}^{n} \begin{Bmatrix}l\\m\end{Bmatrix} \binom{n}{l}(\ln c)^l\mathcal{G}_{n-l,2}^{(k,\alpha)}(-m\ln c;\lambda,a,b)(x)^{(m)}\\
\mathcal{G}_{n,2}^{(k,\alpha)}(x;\lambda, a, b, c)&=\sum_{m=0}^{\infty} \sum_{l=m}^{n} \begin{Bmatrix}l\\m\end{Bmatrix} \binom{n}{l}(\ln c)^l \mathcal{G}_{n-l,2}^{(k,\alpha)}(\lambda,a,b)(x)_{m}\\
\mathcal{G}_{n,2}^{(k,\alpha)}(x;\lambda, a, b)&=\sum_{l=0}^{n}\sum_{m=0}^{n-l}\binom{n}{l}\rsts{l+s}{s}\frac{\binom{n-l}{m}}{\binom{l+s}{s}}\mathcal{G}_{n-l-m}^{(k,\alpha)}(\lambda,a,b)B_m^{(s)}(x\ln c;\lambda)\\
\mathcal{G}_{n,2}^{(k,\alpha)}(x;\lambda, a, b)&=\sum_{m=0}^{\infty} \dfrac{\binom{n}{m}}{(1-\mu)^{s}} \sum_{j=0}^{s} {\binom{s}{j}} (-\mu)^{s-j}\mathcal{G}_{n-m,2}^{(k,\alpha)}(j;\lambda,a,b)F_{m}^{(s)}(x;\mu).
\end{align*}
\end{rema}

\section{Conclusion}
This paper introduced certain variation of poly-Genocchi polynomials, called the Apostol-type poly-Genocchi polynomials of higher order, also known as type 1 Apostol-poly-Genocchi polynomials of higher order, using the concept of polylogarithm and Apostol-type polynomials of higher order with parameters $a$, $b$ and $c$. Some interesting properties and identities of these polynomials were explored parallel to those of the poly-Euler polynomials and poly-Bernoulli polynomials. Using a+ differential identity, the type 1 Apostol-poly-Genocchi polynomials were classified as  Appell polynomials, which, consequently, gave some interesting relations. Moreover, these type 1 Apostol-poly-Genocchi polynomials of higher order were expressed in terms of Stirling numbers of the second kind and Apostol-type poly-Bernoulli polynomials of higher order. Furthermore, the symmetrized generalization of these type 1 Apostol-poly-Genocchi polynomials of higher order was introduced and a kind of double generating function was established. 

\smallskip
The paper was concluded by introducing the type 2 Apostol-poly-Genocchi polynomials of higher order with parameters $a$, $b$ and $c$. Several identities were established, two of which showed the connections of these polynomials with Stirling numbers of the first kind and the type 2 Apostol-type poly-Bernoulli polynomials. One may try to investigate the two types of Apostol-poly-Bernoulli polynomials of higher order defined in \eqref{ApostolPolyBernoulli} and \eqref{ApostolPolyBernoulli-o1}, by establishing more properties and extending them to a more general form by adding three more parameters $a$, $b$ and $c$.


\begin{thebibliography}{99}

\bibitem{ref21} A. Adelberg, Higher Order Bernoulli Polynomials and Newton Polygons. In: Bergum G.E., Philippou A.N., Horadam A.F. (eds) \textit{Applications of Fibonacci Numbers}, Springer, Dordrecht, 1998. https://doi.org/10.1007/978-94-011-5020-0\_1.

\bibitem{ref2} T. Agoh, Convolution identities for Bernoulli and Genocchi polynomials, \textit{Electronic J. Combin.}, \textbf{21} (2014), Article ID P1.65.

\bibitem{ref3} S. Araci, M. Acikgoz, H. Jolany and J.J. Seo, A unified generating function of the $q$-Genocchi polynomials with their interpolation functions, \textit{Proc. Jangjeon Math.Soc.}, \textbf{15}(20)(2012), 227--233.

\bibitem{ref4} S. Araci, Novel identities for $q$-Genocchi numbers and polynomials, \textit{J. Funct.Spaces Appl.}, \textbf{2012}(2012), Article ID 214961.
							
\bibitem{ref5} S. Araci, E. Sen, and M. Acikgoz, Some new formulae for Genocchi Numbers and Polynomials involving Bernoulli and Euler polynomials, \textit{Int. J. Math. Math. Sci.}, \textbf{2014}, Article IC 760613, 7 pages.
							
\bibitem{ref6} S. Araci, Novel identities involving Genocchi numbers and polynomials arising from application of umbral calculus, \textit{Appl. Math. Comput.}, \textbf{233}(2014), 599--607.

\bibitem{ref7} S. Araci, M. Acikgoz and E. Sen, On the von Staudt-Clausen's theorem associated with $q$-Genocchi numbers, \textit{Appl. Math. Comput.}, \textbf{247}(2014), 780--785.

\bibitem{ref8} S. Araci, E. Sen, and M. Acikgoz, Theorems on Genocchi polynomials of higher order arising from Genocchi basis, \textit{Taiwanese J. Math. Math. Sci.}, \textbf{18}(2), 473--482.

\bibitem{ref9} S. Araci, W.A Khan, M. Acikgoz, C. Ozel and P. Kumam, A new generaliztion of Apostol type Hermite-Genocchi polynomials and its applications, \textit{Springerplus}, \textbf{5}(2016), Art. ID 860.

\bibitem{ref10} S. Araci, M. Acikgoz and E. Sen, Some new formulae for Genocchi numbers and polynomials involving Benoulli and Euler polynomials, \textit{Int. J. Math. Sci.}, \textbf{2014}(2014), Article ID 760613.

\bibitem{ref11} T. Arakawa and M. Kaneko, On Poly-Bernoulli Numbers, \textit{Comment.Math. Univ. St. Pauli}, \textbf{48} (1999), 159--167.

\bibitem{ref12} T. Arakawa T. and M. Kaneko, Multi-Zeta Values, Poly-Bernoulli Numbers and Related Zeta Functions, \textit{Nagoya Math. J.}, \textbf{153} (1999), 189--209.

\bibitem{ref13} E. Ayguz, M. Acikgoz and S. Araci, A symmetric identity on the q-Genocchi polynomials of higher-order under third dihedral group D3, \textit{Proc. Jangjeon Math. Soc.}, \textbf{18}(2)(2015), 177--187.

\bibitem{ref15} A. Bayad and Y. Hamahata, Polylogarithms and Poly-Bernoulli Polynomials, \textit{Kyushu J. Math}, \textbf{65}(2011), 15--24.

\bibitem{ref16} A. Bayad and Y. Hamahata, Arakawa-Kaneko L-Functions and Generalized Poly-Bernoulli Polynomials, \textit{J. Number Theory}, \textbf{131}(2011), 1020--1036.

\bibitem{ref18} C. Brewbaker, A Combinatorial Interpretation of the poly-Bernoulli Numbers and two Fermat analogues, \textit{Integers}, \textbf{8}(2008), A02.

\bibitem{ref22} L. Comtet, \textit{Advanced Combinatorics}, Reidel, Dordrecht, The Netherlands, 1974.

\bibitem{ref36} Y. He Y., S. Araci, H.M. Srivastava and M. Acikgoz, Some new identities for the Apostol-Bernoulli polynomials and the Apostol-Genocchi polynomials, \textit{Appl. Math. Comput.}, \textbf{262} (2015), 31-41.

\bibitem{ref37} Y. He, Some new results on products of the Apostol-Genocchi polynomials, \textit{J. Comput. Anal. Appl.}, \textbf{22}(4) (2017), 591-600.

\bibitem{ref38-1} S. Hu, D. Kim, and M.S. Kim, New Identities Involving Bernoulli, Euler and Genocchi Numbers,  \textit{Adv. Difference Equ.}, \textbf{74}(2013).

\bibitem{ref41} H. Jolany H. and R. Corcino, Explicit Formula for Generalization of Poly-Bernoulli Numbers and Polynomials with $a, b, c$ Parameters, \textit{J. Class. Anal.}, \textbf{6}(2015), 119-135.
								
\bibitem{ref42} H. Jolany, M.R. Darafsheh and R.E Alikelaye, Generalizations on Poly-Bernoulli Numbers and Polynomials, \textit{Int. J. Math. Combin.}, \textbf{2} (2010), 7-14.

\bibitem{ref43} M. Kaneko, Poly-Bernoulli Numbers, \textit{J. Theorie de Nombres}, \textbf{9} (1997), 221-228.

\bibitem{ref45-1} D.S. Kim and T. Kim, A note on polyexponential and unipoly functions, \textit{Russ. J. Math. Phys.}, \textbf{26} (2019), 40-49.

\bibitem{ref46} T. Kim, Y.S. Jang and J.J. Seo, A note on poly-Genocchi numbers and polynomials, \textit{Appl. Math. Sci.}, \textbf{8} (2014), 4775-4781. http://dx.doi.org/10.12988/ams.2014.46465.
								
\bibitem{ref46-1} D.S. Kim, D.V. Dolgy, T. Kim, and S.H. Rim, Some Formula for the Product of Two Bernoulli and Euler Polynomials, \textit{Abst. Appl. Anal.}, \textbf{2012}, Article ID 784307, 15 pages.

\bibitem{ref47} T. Kim, S.H. Rim, D.V. Dolgy and S.H. Lee, Some identities of Genocchi polynomials arising from Genocchi basis, \textit{J. Ineq. Appl.}, \textbf{2013} (2013), Article ID 43.

\bibitem{ref47-1} T. Kim, Y.S. Jang and J.J. Seo, Poly-Bernoulli Polynomials and Their Applications, \textit{Int. Journal of Math. Analysis}, \textbf{8}(30) (2014), 1495-1503.

\bibitem{ref49} T. Kim, Some identities for the Bernoulli, the Euler and the Genocchi numbers and polynomials, \textit{Adv. Stud. Contemp. Math.}, \textbf{20}(1) (2010), 23-28.
								
\bibitem{ref51} B. Kurt, Some Identities for the Generalized Poly-Genocchi Polynomials with the Parameters $a$, $b$ and $c$, \textit{J. Math. Anal.}, \textbf{8}(1) (2017), 156-163. 

\bibitem{ref51-1} D.W. Lee, On Multiple Appell Polynomials, \textit{Proc. Amer. Math. Soc.}, \textbf{139} (2011), 2133--2141.

\bibitem{ref54-1} Q.M. Luo, The multiplication formulas for the Apostol-Bernoulli and Apostol-Euler polynomials of higher order, \textit{Integral Transforms Spec. Funct.}, \textbf{20} (2009), 377–391.

\bibitem{ref57-1}  j. Shohat, The Relation of the Classical Orthogonal Polynomials to the Polynomials of Appell, \textit{Amer. J. Math.}, \textbf{58} (1936), 453--464.

\bibitem{ref57-2}  l. Toscano, Polinomi Ortogonali o Reciproci di Ortogonali Nella classe di Appell, \textit{Le Matematiche}, \textbf{11} (1956), 168-174.

\end{thebibliography}
\end{document}